\documentclass[11pt,a4paper]{amsart}
\usepackage{amssymb,amsmath,epsfig,graphics,mathrsfs}

\usepackage{fancyhdr}
\usepackage{hyperref}
\hypersetup{
 colorlinks   = true,
 urlcolor     = blue,
 linkcolor    = blue,
 citecolor   = red ,
 bookmarksopen=true
}

\usepackage{amsmath}
\usepackage{amsfonts}
\usepackage{amssymb}
\usepackage{amsthm}
\usepackage{epsfig,graphics,mathrsfs}
\usepackage{graphicx}

\usepackage[usenames, dvipsnames]{color} 

\usepackage{hyperref}

\newcommand*\MSC[1][1991]{\par\leavevmode\hbox{%
\textit{#1 Mathematical subject classification:\ }}}
\newcommand\blfootnote[1]{%
  \begingroup
  \renewcommand\thefootnote{}\footnote{#1}%
  \addtocounter{footnote}{-1}%
  \endgroup
}

\def \N {\mathbb{N}}

\def \phi {\varphi}

\def \RN {\mathbb{R}^N}
\def \R {\mathbb{R}}

\def \G{\Gamma}

\newcommand{\Rn}{\mathbb R^n}
\newcommand{\Rm}{\mathbb R^m}
\newcommand{\Om}{\Omega}
\newcommand{\Hn}{\mathbb H^n}

\newcommand{\p}{\partial}

\newcommand{\bG}{\mathbb {G}}
\newcommand{\bg}{\mathfrak g}

\newcommand{\la}{\lambda}

\numberwithin{equation}{section}

\newcommand{\beq}{\begin{equation}}
\newcommand{\bea}[1]{\begin{array}{#1} }
\newcommand{\eeq}{ \end{equation}}
\newcommand{\ea}{ \end{array}}

\newcommand{\ve}{\varepsilon}

\newcommand{\In}{\mathbf 1_E}
\newcommand{\Lp}{L^p}

\newtheorem{theorem}{Theorem}[section]

\newtheorem{proposition}[theorem]{Proposition}

\newtheorem{remark}[theorem]{Remark}
\newtheorem{definition}[theorem]{Definition}
\newtheorem{example}[theorem]{Example}

\numberwithin{equation}{section}

\begin{document}

\title[A Bourgain-Brezis-Mironescu-D\'avila theorem, etc.]{A Bourgain-Brezis-Mironescu-D\'avila theorem \\in Carnot groups of step two}

\blfootnote{\MSC[2020]{22E30, 35A08, 53C17}}

\date{}

\begin{abstract}
In this note we prove the following theorem in any Carnot group of step two $\bG$:
\[
\underset{s\nearrow  1/2}{\lim} (1 - 2s) \mathfrak P_{H,s}(E) = \frac{4}{\sqrt \pi}\  \mathfrak P_H(E).
\]
Here, $\mathfrak P_H(E)$ represents the horizontal perimeter of a measurable set $E\subset \bG$, whereas the nonlocal horizontal perimeter $\mathfrak P_{H,s}(E)$ is a heat based Besov seminorm. This result represents a dimensionless sub-Riemannian counterpart of a famous characterisation of Bourgain-Brezis-Mironescu and D\'avila.
\end{abstract}

\author{Nicola Garofalo}

\address{Dipartimento d'Ingegneria Civile e Ambientale (DICEA)\\ Universit\`a di Padova\\ Via Marzolo, 9 - 35131 Padova,  Italy}
\vskip 0.2in
\email{nicola.garofalo@unipd.it}

\thanks{The first author is supported in part by a Progetto SID (Investimento Strategico di Dipartimento) ``Non-local operators in geometry and in free boundary problems, and their connection with the applied sciences", University of Padova, 2017. Both authors are supported in part by a Progetto SID: ``Non-local Sobolev and isoperimetric inequalities", University of Padova, 2019.}

\author{Giulio Tralli}
\address{Dipartimento d'Ingegneria Civile e Ambientale (DICEA)\\ Universit\`a di Padova\\ Via Marzolo, 9 - 35131 Padova,  Italy}
\vskip 0.2in
\email{giulio.tralli@unipd.it}

\maketitle

\tableofcontents

\section{Introduction}\label{S:intro}

In their celebrated papers \cite{BBM1}, \cite{BBM2} (see also \cite{B}) Bourgain, Brezis and Mironescu discovered a new characterisation of the spaces $W^{1,p}$ and BV as suitable limits of the fractional Aronszajn-Gagliardo-Slobedetzky spaces $W^{s,p}$. 
We also mention the earlier work \cite{MN}, in which the authors had already settled the case $p=2$ of their limiting theorem, and the subsequent work \cite{MS}, in which Maz'ya \& Shaposhnikova extended the results in \cite{BBM2}.
Keeping in mind the definition of the seminorm in $W^{s,p}$, see \cite{Ad},
$$
[f]_{p,s} = \left(\int_{\Rn} \int_{\Rn} \frac{|f(x) - f(y)|^p}{|x-y|^{n+ps}} dx dy\right)^{1/p},  \ \ \ \ \ p\ge 1, \ \ \ 0<s<1,  
$$
and denoting by $\In$ the indicator function of a measurable set $E\subset \Rn$, we recall that  such set is said to have finite nonlocal $s$-perimeter if $|E|<\infty$ and 
\begin{equation}\label{psE}
P_s(E) \overset{def}{=} [\mathbf 1_E]^2_{2,s} = [\In]_{1,2s} <\infty.
\end{equation} 
This notion appeared in the above mentioned works  \cite{BBM1}, \cite{BBM2}, \cite{B}, as well as in Maz'ya's paper \cite{Ma}, and in the work of Caffarelli, Roquejoffre and Savin \cite{CRS}, in which these authors have first studied the  Plateau problem for the relevant nonlocal minimal surfaces. It is well-known that every non-empty open set has infinite $s$-perimeter as soon as $1/2\le s<1$. For instance, if we denote by $B = \{x\in \Rn\mid |x|<1\}$, then it was observed in \cite{Glincei} that
\[
P_s(B) =   \frac{n \pi^{n}  \G(1-2s)}{s \G(\frac n2+1) \G(1-s)\G(\frac{n+2-2s}{2})},
\]
where for $x>0$ we have indicated by $\G(x) = \int_0^\infty t^{x-1} e^{-t} dt$ the Euler gamma function.
It is clear from this formula, as well as from those appeared in \cite{FS} and \cite{FFMMM}, that $s\to P_s(B)$ has a  pole in $s = 1/2$ (and also in $s = 0$), and that moreover one has the limiting relation
\[
\underset{s\nearrow 1/2}{\lim} (1-2s) P_s(B) = \frac{2 \pi^{\frac{n-1}2}}{\G(\frac{n+1}{2})}\  P(B),
\]
where we have denoted by $P(B) = \frac{2 \pi^{\frac n2}}{\G(\frac n2)}$ the standard perimeter of $B$. This limit relation is in fact a special case of the following result proved by J. D\'avila in \cite[Theor. 1]{Davila}, and conjectured in \cite{BBM1}:
\begin{equation}\label{dav}
\underset{s\nearrow 1/2}{\lim}\ (1 -2s) P_s(E) = \left(\int_{\mathbb S^{n-1}} |<e_n,\omega>|d\sigma(\omega)\right)\ P(E),
\end{equation}
where $e_n = (0,...,0,1)$, and $P(E)$ indicates the perimeter of $E$ according to De Giorgi, see \cite{DG54}. The limiting behaviour of the fractional perimeter was also studied in \cite{ADM} and \cite{CV}.
All these results underscore an important aspect of the nonlocal minimal surfaces: they asymptotically converge to the classical ones.

To introduce the results in the present paper, we now make the crucial observation that theorem \eqref{dav} admits a dimension-free formulation using the heat semigroup $P^\Delta_t f(x) = e^{-t \Delta} f(x) = (4\pi t)^{-n/2} \int_{\Rn} e^{-\frac{|x-y|^2}{4t}} f(y) dy$. For $s>0$ and $1\le p<\infty$, consider the following caloric Besov seminorm
\begin{equation}\label{ndelta}
\mathscr N^\Delta_{s,p}(f) = \left(\int_0^\infty  \frac{1}{t^{\frac{s p}2 +1}} \int_{\Rn} P^\Delta_t\left(|f - f(x)|^p\right)(x) dx dt\right)^{\frac 1p}.
\end{equation}
Seminorms such as \eqref{ndelta} were first considered by Taibleson in his works \cite{T1}, \cite{T2} for quite different purposes than those in the present note. 
We leave it as an easy exercise for the reader to recognise that 
\begin{equation}\label{equiv}
\mathscr N^\Delta_{s,p}(f)^p = \frac{2^{sp} \G(\frac{n+sp}2)}{\pi^{\frac n2}}\ [f]_{s,p}^p.
\end{equation}
Combining \eqref{equiv} with \eqref{psE} and \eqref{dav}, and keeping in mind that $
\int_{\mathbb S^{n-1}} |<e_n,\omega>| d\sigma(\omega)= \frac{2\pi^{\frac{n-1}2}}{\G(\frac{n+1}2)}$, we now see that the theorem of Bourgain-Brezis-Mironescu-D\'avila  can be reformulated in terms of the heat seminorm \eqref{ndelta} in the following suggestive dimension-free fashion: 
\begin{equation}\label{dav2}
\underset{s\nearrow 1/2}{\lim}\ (1 -2s) \mathscr N^\Delta_{2s,1}(\In) = \underset{s\nearrow 1/2}{\lim}\ (1 -2s) \mathscr N^\Delta_{s,2}(\In)^2 = \frac{4}{\sqrt \pi} \ P(E).
\end{equation}

The present work stems from the desire of understanding what happens to \eqref{dav2} if we leave the Euclidean setting and move into sub-Riemannian geometry. Does the Bourgain-Brezis-Mironescu-D\'avila phenomenon persist? Our main result proves that, remarkably, the answer is yes - and with exactly the same universal constant as in \eqref{dav2}! - in the framework of stratified nilpotent Lie groups of step two provided that:
\begin{itemize}
\item[(i)] the perimeter of De Giorgi $P(E)$ is replaced by the sub-Riemannian horizontal perimeter in \eqref{hper} below;
\item[(ii)] the nonlocal perimeter $P_s(E)$ in \eqref{psE} is replaced by a notion of nonlocal horizontal perimeter defined via some Besov seminorms based  on the heat semigroup in $\bG$, see Definition \ref{D:phs} below. 
\end{itemize}

Before stating our main theorem we mention that the basic prototype of the geometric ambients in this note is the Heisenberg group $\Hn$, which is the primary model of a Sasakian CR manifold with zero Tanaka-Webster Ricci tensor, see \cite{DT}. More generally, a distinguished subclass of step-two Carnot groups is formed by the groups of Heisenberg type which were introduced by Kaplan in \cite{Ka} in connection with hypoellipticity questions. Groups of Heisenberg type constitute a direct and important generalisation of the Heisenberg group, as they include, in particular,  Iwasawa groups, i.e., the nilpotent component $N$ in the Iwasawa decomposition $KAN$ of simple groups of rank one, see in this respect the seminal work of Cowling, Dooley, Kor\'anyi and Ricci \cite{CDKR}, and also the visionary address of E. Stein \cite{Snice} at the 1970 International Congress of Mathematicians in Nice. 

Given a Carnot group of step two $\bG$, we indicate with $\mathfrak P_H(E)$ the horizontal perimeter of a set $E\subset \bG$. Such notion was introduced in \cite{CDGcag} in much greater generality than Lie groups, and a corresponding theory of isoperimetric inequalities was subsequently developed in \cite{GNcpam}. Next, let us denote with $\mathfrak P_{H,s}(E)$ the key notion of nonlocal horizontal $s$-perimeter of $E$ as in Definition \ref{D:phs} below. The following is the main theorem of this note.

\begin{theorem}\label{T:main}
Let $\bG$ be a Carnot group of step two, and let $E\subset \bG$ be a measurable set having finite horizontal perimeter and such that $|E|<\infty$. Then,
\[
\underset{s\nearrow  1/2}{\lim} (1 - 2s) \mathfrak P_{H,s}(E) = \frac{4}{\sqrt \pi}\  \mathfrak P_H(E).
\]
\end{theorem}

One should compare Theorem \ref{T:main} with the dimension-free version of the Bourgain-Brezis-Mironescu and D\'avila theorem in \eqref{dav2} above. It is worth mentioning explicitly that our result underscores the critical role of the heat based notion of nonlocal perimeter in Definition \ref{D:phs}. Our proof of Theorem \ref{T:main} combines the interesting Ledoux type result in the work of Bramanti, Miranda and Pallara \cite[Theorem 2.14]{BMP} with two crucial asymptotic estimates for the nonlocal perimeter $\mathfrak P_{H,s}(E)$ which are proved in Section \ref{S:two}. We emphasise that such estimates continue to be valid in Carnot groups of arbitrary step and, more in general, for operators of H\"ormander type under suitable assumptions. We mention that even for $\Rn$ our proof provides a new perspective on \eqref{dav} (based also on the result in \cite{MPPP}), with a dimensionless constant in the right-hand side.

Having stated our main result we briefly describe the organisation of the paper. In Section \ref{S:prelim} we collect some basic facts which will be needed in the rest of the paper. In Section \ref{S:two} we introduce the notion of nonlocal horizontal $s$-perimeter, see Definition \ref{D:phs}. Then, we prove the two key results of the section, Propositions \ref{P:star} and  \ref{P:starsotto}. These two results allow us to conclude in Section \ref{S:main} that the limit in the left hand-side of the equation in Theorem \ref{T:main} does exist, and moreover
\begin{equation}\label{uhuh}
\underset{s\nearrow  1/2}{\lim} (1 - 2s) \mathfrak P_{H,s}(E) = \underset{t \to 0^+}{\lim}\ \sqrt{\frac{4}{t}}\ ||P_t\In - \In||_{L^1(\bG)},
\end{equation}
provided that the limit of the sub-Riemannian Ledoux functional in the right-hand side of \eqref{uhuh} exists.
At this point, we exploit the geometric measure theoretic result in \cite[Theorem 2.14]{BMP}, which states that in every Carnot group of step two one has
\begin{equation}\label{ohoh}
\underset{t \to 0^+}{\lim}\ \sqrt{\frac{4}{t}}\ ||P_t\In - \In||_{L^1(\bG)} = 8 \int_{\bG} \phi_\bG(\nu_E) d |\p_H E|,
\end{equation}
where we have indicated by $d |\p_H E|$ the horizontal perimeter measure, and by $\phi_\bG(\nu_E)$ the function defined in (26) of \cite{BMP}. Finally, we prove that the function $\phi_\bG(\nu_E)$ is a universal constant which is independent of both the dimension of the horizontal layer and of that of the vertical layer of $\bG$. Once this is recognised, the proof of Theorem \ref{T:main} follows.

In closing we mention the recent work \cite{BGT} in which we have studied the limiting behaviour as $s\to 0^+$ of Besov seminorms such as \eqref{sn} below, but associated to some non-symmetric and non-doubling semigroups whose generators contain a drift. We also mention two upcoming works 
that are connected to the present one. 
In the former \cite{GTisoG}, we develop in the setting of arbitrary Carnot groups some optimal nonlocal isoperimetric inequalities which involve the notion of nonlocal horizontal perimeter in Definition \ref{D:phs}. In a related perspective, but with a different framework, the reader should also see \cite{GTiso}. In the latter \cite{GTledG}, we provide in the setting of the Heisenberg group $\Hn$ a stronger version of the sub-Riemannian Ledoux limiting relation in \eqref{ohoh}. Our result extends in a nontrivial way a preceding result in \cite{MPPP} in the Euclidean setting, see also the original paper by Ledoux \cite{Led}, where this circle of ideas originated.


\section{Preliminaries}\label{S:prelim}

In this section we collect some preliminary material that will be used in the rest of the paper. We begin with introducing the main geometric ambients of this note. A Carnot group of step $r = 2$ is a simply-connected Lie group $\bG$ whose Lie algebra admits a stratification $\bg = V_1 \oplus V_2$, with $[V_1,V_1] = V_2$ and $[V_1,V_2] = \{0\}$. We let $m=\operatorname{dim}(V_1)$ and $k = \operatorname{dim}(V_2)$. Assuming that $\bg$ is endowed with an inner product $\langle \cdot,\cdot\rangle$ and induced norm $|\cdot|$, then the Kaplan mapping $J: V_2 \to \operatorname{End}(V_1)$ defined by 
$$
<J(\sigma)z,z'> = <[z,z'],\sigma>
$$
has the properties that $J(\sigma)^\star = - J(\sigma)$ and $J$ is an injective map which is linear as a function of $\sigma$. Thus the mapping 
\begin{equation}\label{defA}
A(\sigma) \overset{def}{=} J^\star(\sigma) J(\sigma) = - J(\sigma)^2,
\end{equation}
defines a symmetric nonnegative element of $\operatorname{End}(V_1)$ for every $\sigma \in V_2$.

\begin{example}\label{exHtipo}
A Carnot group of step two is said of Heisenberg type if $J(\sigma)$ is orthogonal for every $\sigma\in V_2$ such that $|\sigma| = 1$. We refer the reader to \cite{Ka}, \cite{CDKR} and \cite[Section 2]{GParis} for an extensive discussion. In particular, when $\bG$ is of Heisenberg type, we have $m=2n$ for some $n\in \N$ and
\begin{equation}\label{AHtipo}
A(\sigma)=|\sigma|^2 \mathbb{I}_{2n} \quad\mbox{ for all }\sigma\in V_2.
\end{equation}
If in addition $k=1$ then the group $\bG$ boils down, up to isomorphism, to the Heisenberg group $\Hn$.
\end{example}

We fix orthonormal basis $\{e_1,...,e_{m}\}$ and $\{\ve_1,...,\ve_k\}$ for $V_1$ and $V_2$ respectively, and for points $z\in V_1$ and $\sigma\in V_2$ we will use either one of the representations $z = \sum_{j=1}^{m} z_j e_j$, $\sigma = \sum_{\ell=1}^k \sigma_\ell \ve_\ell$, or also $z = (z_1,...,z_{m})$, $\sigma = (\sigma_1,...,\sigma_k)$. Accordingly, whenever convenient we will identify the point $g = \exp(z+\sigma)\in \bG$ with its logarithmic coordinates $(z,\sigma)$. 
By the Baker-Campbell-Hausdorff formula, see p. 12 of \cite{CGr}, 
\begin{equation}\label{bch}
\exp(z+\sigma)  \exp(z'+\sigma') = \exp{\big(z + z' + \sigma +\sigma' + \frac{1}{2}
[z,z']\big)},
\end{equation}
we obtain the non-Abelian 
multiplication in $\bG$
$$
g\circ g' = (z +z',\sigma + \sigma' + \frac 12 \sum_{\ell=1}^{k} <J(\ve_\ell)z,z'>\ve_\ell).
$$
If for $j=1,...,m$ we define left-invariant vector fields by the Lie rule $X_j u(g) = \frac{d}{ds} u(g \circ \exp s e_j)\bigg|_{s=0}$,
then by \eqref{bch} one obtains in the logarithmic coordinates $(z,\sigma)$ 
\begin{equation}\label{Xi}
X_j  = \p_{z_j} + \frac 12 \sum_{\ell=1}^k <J(\ve_\ell)z,e_j> \partial_{\sigma_\ell}.
\end{equation}
Given a function $f\in C^1(\bG)$ we will indicate by 
$\nabla_H f = (X_1f,...,X_{m} f)$,
its horizontal gradient, and set $|\nabla_H f| = (\sum_{j=1}^{m} (X_j f)^2)^{1/2}$.

\subsection{Horizontal perimeter}
We next recall the variational notion of horizontal perimeter that in our main result will replace the perimeter of De Giorgi $P(E)$ in \eqref{dav}. Such notion was introduced in \cite{CDGcag} (in much greater generality than in stratified nilpotent Lie groups) and it has since occupied a central position in the theory of minimal surfaces in sub-Riemannian geometry. Given an open set $\Om\subset \bG$, we denote by 
\[
\mathscr F(\Om) = \{\zeta=(\zeta_1,...,\zeta_{m})\in C_0^1(\Om;\R^{m})\big|\ ||\zeta||_{\infty} = \underset{g\in \Om}{\sup} \big(\sum_{j=1}^{m} \zeta_j(g)^2\big)^{1/2} \le 1\}.
\]
We say that a function $f\in L^1_{loc}(\Om)$ has bounded horizontal total variation in $\Om$ if
\[
\operatorname{Var}_H(f,\Om) \overset{def}{=} \underset{\zeta\in \mathscr F(\Om)}{\sup} \int_{\Om} f \sum_{j=1}^{m} X_j \zeta_j dg < \infty.
\] 
Hereafter, we denote with $dg$ the bi-invariant Haar measure in $\bG$ obtained by pushing forward with the exponential map the standard Lebesgue measure in the Lie algebra $\bg$. Such Haar measure interacts with the group non-isotropic dilations, $\delta_\la(z,\sigma) = (\la z,\la^2 \sigma)$, according to the formula
\begin{equation}\label{Q}
d \delta_\la(g) = \la^Q dg,
\end{equation}
where $Q = m + 2k$ is the group homogeneous dimension, see \cite{FShardy}. 
The Banach space of $L^1(\Om)$ functions of bounded horizontal total variation in $\Om$, with its norm given by 
\[
||f||_{\operatorname{BV}_H(\Om)} = ||f||_{L^1(\Om)} + \operatorname{Var}_H(f,\Om),
\]
 will be denoted by $\operatorname{BV}_H(\Om)$. Given a set $E\subset \bG$ such that $|E|<\infty$, we say that $E$ has finite horizontal perimeter with respect to $\Om$ if $\In \in \operatorname{BV}_H(\Om)$. If this is the case, the horizontal perimeter of $E$ in $\Om$ is defined as
\begin{equation}\label{hper}
\mathfrak P_H(E;\Om) =  \operatorname{Var}_H(\In,\Om) = \underset{\zeta\in \mathscr F(\Om)}{\sup} \int_{E\cap\Om}  \sum_{j=1}^{m} X_j \zeta_j dg.
\end{equation}
When $\Om = \bG$ we will simply write $\mathfrak P_H(E)$, instead of  $\mathfrak P_H(E;\bG)$. For the main properties of the space $\operatorname{BV}_H$ and general sharp isoperimetric inequalities for the horizontal perimeter, we refer the reader to \cite{GNcpam}.

\subsection{Heat semigroup}

The horizontal Laplacian  generated by the orthonormal basis $\{e_1,...,e_{m}\}$ of $V_1$ is the second-order differential operator on $\bG$ defined by 
$$
\Delta_H f = \sum_{j=1}^{m} X_j^2 f,
$$
where $X_1,...,X_{m}$ are given by \eqref{Xi}. From \eqref{Xi} one finds
\begin{equation}\label{slH}
\Delta_H = \Delta_z + \frac 14 \sum_{\ell,\ell' = 1}^k \langle J(\ve_\ell)z,J(\ve_{\ell'})z\rangle \p_{\sigma_\ell}\p_{\sigma_{\ell'}} + \sum_{\ell = 1}^k \Theta_\ell \p_{\sigma_\ell},
\end{equation}
where $\Delta_z$ represents the standard Laplacians in the variables $z = (z_1,...,z_{m})$ and
\[
\Theta_\ell = \sum_{i=1}^{m} <J(\ve_\ell)z,e_i> \p_{z_i}.
\]
The operator $\Delta_H$ fails to be elliptic at every point $g\in \bG$. However, thanks to the commutation relation $[X_i,X_{j}] = \sum_{\ell=1}^k \langle J(\ve_\ell) e_i,e_j\rangle \p_{\sigma_\ell}$ (which follows from \eqref{Xi}),  and to the fundamental hypoellipticity theorem of H\"ormander in \cite{Ho}, one knows that $\Delta_H$ is hypoelliptic, see also \cite{Fofs}.

In \cite{Fo} Folland proved, for stratified nilpotent Lie groups $\mathbb G$ of arbitrary step, the existence of a fundamental solution $p(g,g',t)$ for the heat equation $\p_t u - \Delta_H u = 0$ associated with a horizontal Laplacian on $\mathbb G$. We need the heat semigroup defined by 
\begin{equation}\label{pt}
P_t u(g) = \int_{\bG} p(g,g',t) u(g') dg'.
\end{equation}
It is well-known that \eqref{pt} defines a stochastically complete, positive and symmetric semigroup in $\Lp(\bG)$, for any $1\le p\le \infty$, which is contractive
\begin{equation}\label{pcontr}
||P_t u||_{\Lp(\bG)} \le ||u||_{\Lp(\bG)},\ \ \ \ \ \ \ 1\le p \le \infty.
\end{equation}
Although we will not make explicit use of the following Gaussian estimates (which are corollaries of the general results in \cite{JS} and \cite{KS}) we state them for completeness
\begin{equation}\label{ge}
C t^{- Q/2} \exp(- \alpha \frac{d(g,g')^2}{t}) \le p(g,g',t) \le C^{-1} t^{- Q/2} \exp(- \beta \frac{d(g,g')^2}{t}).
\end{equation}
In \eqref{ge} we have indicated with $d(g,g')$ the left-invariant intrinsic distance in $\bG$ defined by 
$d(g,g') \overset{def}{=} \sup \{f(g) - f(g')\mid f\in C^\infty(\bG),\ |\nabla_H f|\le 1\}$. 
Such $d(g,g')$ coincides with the Carnot-Carath\'eodory distance defined in \cite{NSW}. If we indicate by $B(g,r) =\{g'\in \bG\mid d(g,g')<r\}$, then by scale invariance we obtain for every $g\in \bG$ and $r>0$, $|B_\rho(g,r)| = \omega\ r^Q$,
where $\omega>0$ is a universal constant and $Q$ is as in \eqref{Q}, and this accounts for the term $t^{- Q/2}$ in \eqref{ge}. 

In the proof of Theorem \ref{T:main} it is of paramount importance to have a flexible formula for the heat kernel in a Carnot group of step two. We recall that in the Heisenberg group $\Hn$ an explicit formula, up to Fourier transform in the central variable, was first independently discovered by Hulanicki \cite{Hu} and Gaveau \cite{Gav}, and subsequently generalised to groups of Heisenberg type in \cite{Cy, Ran}. For general Carnot groups of step two, the heat kernel was constructed by Cygan in \cite{Cy} with the aid of a lifting procedure, and more recently by Beals, Gaveau and Greiner \cite{BGG} using complex Hamiltonians. We shall use Theorem \ref{T:heat} below which represents a version of Cygan's formula which is tailor made for our purposes. We have recently obtained this result in \cite{GTstep2} with a new approach based on the Ornstein-Uhlenbeck operator.

We denote by $\sqrt{A(\la)}$ the square root of the nonnegative matrix defined in \eqref{defA}. Moreover, given a $m\times m$ symmetric matrix $M$ with real coefficients, we denote by $j(M)$ the matrix identified by the power series of the real-analytic function $j:\R\to (0,1]$ given by $j(x) = \frac{x}{\sinh x}$. An analogous interpretation holds for the matrix $\cosh M$. 

\begin{theorem}\label{T:heat}
Let $\bG$ be a Carnot group of step two. For $g=(z,\sigma), g' = (z',\sigma') \in \bG$ and $t>0$, the heat kernel relative to the horizontal Laplacian in \eqref{slH} is then given by
\begin{align}\label{uffa}
& p(g,g',t) = 2^k (4\pi t)^{-(\frac m2 + k)}  \int_{\R^k} e^{\frac{i}{t}\left(\langle \sigma'-\sigma,\la\rangle +\frac{1}{2}\langle J(\la)z',z\rangle \right)} \left(\det j(\sqrt{A(\la)})\right)^{1/2} 
\\
& \times \exp\bigg\{-\frac{1}{4t}\langle j(\sqrt{A(\la)}) \cosh \sqrt{A(\la)} (z-z'),z-z'\rangle \bigg\}
d\la.\notag
\end{align}
\end{theorem}

We mention that in \eqref{uffa} we have identified the vertical layer $V_2\subset \bg$ with $\R^k$.
When $\bG$ is a group of Heisenberg type as in Example \ref{exHtipo}, we have $\sqrt{A(\la)}=|\la|\mathbb{I}_{2n}$ and we  easily recover the expression
\begin{equation}\label{uffaHtipo}
\frac{2^k}{\left(4\pi t\right)^{n+k}}\int_{\R^k} e^{\frac it (\langle\sigma' -\sigma,\la \rangle + \frac 12\langle J(\la)z',z\rangle)} \left(\frac{|\la|}{\sinh |\la|}\right)^n  e^{- \frac{|z-z'|^2}{4t} \frac{|\la|}{\tanh |\la|}}d\la,
\end{equation}
which is the formula of Hulanicki and Gaveau\footnote{We warn the reader that different authors choose different group laws and normalisations. As a consequence, in the cited works formula \eqref{uffaHtipo} appears with different constants.}.


\section{The nonlocal horizontal $s$-perimeter and two key asymptotics}\label{S:two}

With the notion of horizontal perimeter in hands, we now turn to the second key player in our main result: the nonlocal horizontal perimeter. In this section, using a new functional space based on the heat semigroup, we introduce this object in Definition \ref{D:phs}. Then, we prove two results that will play a key role in the proof of Theorem \ref{T:main}. The former provides an interesting one-sided bound for the limiting case $s=1/2$ similar to the right-hand side of the Bourgain, Brezis and Mironescu's bound in \cite{BBM1}. The latter instead contains a lower bound for such limit. The results in this section are valid without changes in much greater generality than groups of step two since, as it will be clear with the proofs, they rely on the stochastic completeness and the contractive properties of the relevant semigroup. As we have mentioned in the introduction, Propositions \ref{P:star} and \ref{P:starsotto} hold in fact in Carnot groups of arbitrary step, as well as for general H\"ormander type operators under suitable hypotheses.

Consider the heat semigroup \eqref{pt}. For any $0<s<1$ and $1\le p<\infty$ we define
the horizontal Besov space $\mathfrak B_{s,p}(\bG) = \mathfrak B^{\Delta_H}_{s,p}(\bG)$ as the collection of all functions $u\in L^p(\bG)$ such that the seminorm
\begin{equation}\label{sn}
\mathscr N_{s,p}(u) = \left(\int_0^\infty  \frac{1}{t^{\frac{s p}2 +1}} \int_{\bG} P_t\left(|u - u(g)|^p\right)(g) dg dt\right)^{\frac 1p} < \infty.
\end{equation}
The norm $||u||_{\mathfrak B_{s,p}(\bG)} = ||u||_{\Lp(\bG)} + \mathscr N_{s,p}(u)$ turns $\mathfrak B_{s,p}(\bG)$ into a Banach space. We emphasise that  $\mathfrak B_{s,p}(\bG)$ is nontrivial since, for instance, it contains $C^\infty_0(\bG)$. In a different, but related perspective, similar semigroup based Besov spaces have been introduced and developed in \cite{GTfi}, \cite{GTiso} and \cite{BGT}. We mention that in the framework of stratified nilpotent Lie groups the limiting behaviour in $s$ of fractional Aronszajn-Gagliardo-Slobedetzky seminorms different from \eqref{sn} was studied in \cite{MP}, without any identification of the sharp constants.  
We are ready to introduce the central notion in this note.

\begin{definition}\label{D:phs}
Given $0<s<1/2$, we say that a measurable set $E\subset \bG$ has finite horizontal $s$-perimeter if $\In\in \mathfrak B_{2s,1}(\bG)$, and we define
\[
\mathfrak P_{H,s}(E) \overset{def}{=} \mathscr N_{2s,1}(\In)  < \infty.
\]
We call the number $\mathfrak P_{H,s}(E)\in [0,\infty)$ the  horizontal $s$-perimeter of $E$ in $\bG$.
\end{definition}

We have the following result.

\begin{proposition}\label{P:star}
For every measurable set $E\subset \bG$, such that $|E|<\infty$, one has
\begin{equation}\label{ve2}
\underset{s\nearrow 1/2}{\limsup}\ (1 - 2s)\ \mathfrak P_{H,s}(E)   \le  \underset{t \to 0^+}{\limsup}\ \sqrt{\frac{4}{t}}\ ||P_t\In - \In||_{L^1(\bG)}.
\end{equation}  
\end{proposition}

\begin{proof}
Henceforth in this note, whenever there is no risk of confusion we simply write $||\cdot||_1$, instead of $||\cdot||_{L^1(\bG)}$. Now, if $L \overset{def}{=}  \underset{t \to 0^+}{\limsup}\ \frac{1}{\sqrt{t}} ||P_t\In - \In||_1  = \infty$, then \eqref{ve2} is trivially valid. Therefore, we might as well assume that $L<\infty$. This implies the existence of $\ve_0>0$ such that 
\[
\underset{\tau\in (0,\ve_0)}{\sup}\ \frac{1}{\sqrt{\tau}} ||P_\tau\In - \In||_1 \le L+1.
\]
For every $0<s<1/2$ we thus have 
\begin{align*}
& \int_0^{\ve_0} \frac{1}{\tau^{1+s}} ||P_\tau\In - \In||_1 \le \underset{\tau\in (0,\ve_0)}{\sup}\ \frac{1}{\sqrt{\tau}} ||P_\tau\In - \In||_1 \int_0^{\ve_0} \frac{d\tau}{\tau^{1+s-1/2}}
 \le (L+1) \frac{\ve_0^{1/2-s}}{1/2 -s} < \infty.
\end{align*} 
On the other hand, since by \eqref{pcontr} $P_\tau$ is contractive on every $\Lp(\bG)$, with $1\le p\le \infty$, we have
\[
\int_{\ve_0}^\infty \frac{1}{\tau^{1+s}} ||P_\tau\In - \In||_1 d\tau \le 2 |E| \int_{\ve_0}^\infty \frac{d\tau}{\tau^{1+s}} < \infty.
\] 
The latter two inequalities show that
\begin{equation}\label{ukay}
\int_{0}^\infty \frac{1}{\tau^{1+s}} ||P_\tau\In - \In||_1 d\tau <\infty.
\end{equation}
Using the stochastic completeness of the semigroup, we can recognise by Definition \ref{D:phs} and \eqref{sn} that
\begin{align}\label{uguali}
\mathfrak P_{H,s}(E) &= \int_0^\infty  \frac{1}{t^{s  +1}} \int_{\bG} P_t\left(|\In - \In(g)|\right)(g) dg dt\\
&= \int_0^\infty  \frac{1}{t^{s  +1}} \left(\int_{\bG\smallsetminus E}P_t(\In) (g) dg + \int_{E}(1- P_t(\In) (g))dg \right)dt\nonumber\\
&= \int_0^\infty  \frac{1}{t^{s  +1}} \left(\int_{\bG\smallsetminus E}\left|P_t(\In) (g) -\In(g)\right| dg + \int_{E}\left|\In(g) -P_t(\In) (g)\right|dg \right)dt\nonumber\\
&=\int_{0}^\infty \frac{1}{t^{1+s}} ||P_t\In - \In||_1 dt < \infty,
\nonumber
\end{align}
where in the last inequality we have used \eqref{ukay}. Thus, assuming $L<+\infty$, we reach the conclusion that the set $E$ has finite horizontal $s$-perimeter for every $0<s<\frac{1}{2}$. With this being said, given any $\ve\in (0,\ve_0)$, we now obtain from \eqref{uguali}
\begin{align*}
\mathfrak P_{H,s}(E) & = \int_0^\ve \frac{1}{\tau^{1+s}} ||P_\tau\In - \In||_1 d\tau +  \int_\ve^\infty \frac{1}{\tau^{1+s}} ||P_\tau\In - \In||_1 d\tau.
\end{align*}
As before, one easily recognises
\[
\int_\ve^\infty \frac{1}{\tau^{1+s}} ||P_\tau\In - \In||_1 d\tau \le \frac{2 |E|}{s}\ \ve^{-s}.
\] 
On the other hand, one has
\[
\int_0^\ve \frac{1}{\tau^{1+s}} ||P_\tau\In - \In||_1 d\tau \le \underset{\tau\in (0,\ve)}{\sup}\ \frac{1}{\sqrt \tau} ||P_\tau\In - \In||_1\ \frac{\ve^{1/2 - s}}{1/2 - s}.
\]
We infer that for every $\ve\in (0,\ve_0)$ we have
\begin{equation}\label{ve}
\mathfrak P_{H,s}(E) \le \frac{1}{\left(1/2 - s\right)}\ \underset{\tau\in (0,\ve)}{\sup}\ \frac{1}{\sqrt \tau} ||P_\tau\In - \In||_1\ \ve^{1/2 - s} + \frac{2 |E|}{s}\ \ve^{-s}.
\end{equation}
Multiplying by $(1 - 2s)$ in \eqref{ve} and taking the limit as $s\nearrow 1/2$, we find for any $\ve\in (0,\ve_0)$,
\begin{equation}\label{lim12}
\underset{s\nearrow 1/2}{\limsup} (1 - 2s) \mathfrak P_{H,s}(E)  \le \underset{\tau\in (0,\ve)}{\sup}\ \sqrt{\frac 4\tau}\ ||P_\tau\In - \In||_1.
\end{equation}
Passing to the limit as $\ve\to 0^+$ in \eqref{lim12}, we reach the desired conclusion \eqref{ve2}. 

\end{proof}

Our next result can be seen as dual to Proposition \ref{P:star}.

\begin{proposition}\label{P:starsotto}
For every measurable set $E\subset \bG$  one has
$$
\underset{s\nearrow 1/2}{\liminf}\ (1 - 2 s)\ \mathfrak P_{H,s}(E)   \ge  \underset{t \to 0^+}{\liminf}\ \sqrt{\frac 4 t}\ ||P_t \In - \In||_{L^1(\bG)}.
$$
\end{proposition}

\begin{proof}
Before starting we remark that \eqref{uguali} holds regardless the assumption that the right-hand side be finite. Now, for every $0<s<1/2$ and any $\ve>0$, the identity \eqref{uguali} yields
\begin{align*}
& (1-2s)  \mathfrak P_{H,s}(E) \ge (1-2s) \int_0^\ve \frac{1}{t^{1+s}} ||P_t\In - \In||_{1} dt
\\
& \ge (1 -2s) \underset{0<t<\ve}{\inf}\  \frac{1}{\sqrt t} ||P_t\In - \In||_{1} \int_0^\ve t^{1/2-s-1} dt
\\
& =  \underset{0<t<\ve}{\inf}\ \sqrt{\frac{4}{t}} ||P_t\In - \In||_{1}\ \ve^{1/2-s}. 
\end{align*}
Taking the $\liminf$ as $s\nearrow 1/2$ in the latter inequality, gives
\[
\underset{s\nearrow 1/2}{\liminf} (1-2s) \mathfrak P_{H,s}(E) \ge \underset{0<t<\ve}{\inf}\ \sqrt{\frac 4t} ||P_t\In - \In||_{1}.
\]
If we now take the limit as $\ve\to 0^+$, we reach the desired conclusion.

\end{proof}



\section{Proof of Theorem \ref{T:main}}\label{S:main}

In this section we finally prove Theorem \ref{T:main}. Our first key observation is that, if we knew that for any measurable set $E\subset \bG$ with finite horizontal perimeter we have
\begin{equation}\label{mira}
\underset{t \to 0^+}{\liminf}\ \sqrt{\frac 4t} ||P_t\In - \In||_1 = \underset{t \to 0^+}{\limsup}\ \sqrt{\frac 4t} ||P_t\In - \In||_1 = \frac{4}{\sqrt \pi} \mathfrak P_H(E),
\end{equation}
then the combination of Propositions \ref{P:star}, \ref{P:starsotto} and \eqref{mira} would give 
\[
\underset{s\nearrow 1/2}{\lim} (1-2s) \mathfrak P_{H,s}(E) = \frac{4}{\sqrt \pi} \mathfrak P_H(E),
\]
and Theorem \ref{T:main} would be proved. To establish \eqref{mira} we first appeal to (29) in \cite[Theorem 2.14]{BMP} which ensures that the limit in \eqref{mira} does exist for any set of finite horizontal perimeter $E\subset \bG$, and moreover
\begin{equation}\label{miramira} 
\underset{t \to 0^+}{\lim}\ \sqrt{\frac 4t} ||P_t\In - \In||_1 = 8 \int_\bG \phi_\bG(\nu_E) d|\p_E|. 
\end{equation} 
In the right-hand side of \eqref{miramira} for every horizontal unit vector $\nu$ the function $\phi_\bG$ is defined by
\begin{equation}\label{phi}
\phi_\bG(\nu) = \int_{T_\bG(\nu)} p(\hat g,e,1) d\hat g,
\end{equation}
where we have denoted by $e\in \bG$ the identity and by $\hat g$ the generic point on the vertical space $T_\bG(\nu)$ perpendicular to the horizontal vector $\nu$, see (26) in \cite{BMP}. 

\begin{remark}
It is easy to see from \eqref{uffaHtipo} that in the particular case of a group of Heisenberg type $\bG$ the function $\phi_\bG(\nu)$ must be independent of the horizontal vector $\nu$ and therefore constant, see in this respect \cite[Remark 2.12]{BMP}. We emphasise that this circumstance is by no means enough to complete the proof of Theorem \ref{T:main} since the latter hinges crucially on the identification of such constant value.
\end{remark}

To this task we now turn and we claim that, remarkably, in a general Carnot group of step two $\bG$, the function $\phi_\bG(\nu)$ in \eqref{phi} is also a universal constant and we have in fact
\begin{equation}\label{wow}
\phi_\bG(\nu)\equiv \frac{1}{\sqrt{4\pi}}\quad\mbox{ for every horizontal unit vector }\nu.
\end{equation}
If we take this claim for granted, then in light of the above discussions and recalling that $\mathfrak P_H(E)=\int_\bG d|\p_E|$, it is immediate to finish the proof of Theorem \ref{T:main} by inserting \eqref{wow} in \eqref{miramira}. We are thus left with the proof of the claim in \eqref{wow}.

As we have so far done in this paper, we keep identifying $V_1 \cong \Rm$, $V_2 \cong \R^k$. Given a unit vector $\nu\in \Rm$ we denote by $T_\nu$ the $(m-1)$-dimensional subspace of $\R^m$ defined as $\left(\rm{span}\{\nu\}\right)^{\perp}$, and indicate with $P_\nu:\Rm\to \Rm$ the orthogonal projection onto $T_\nu$, i.e. $P_\nu z=z-<z,\nu>\nu$. Clearly, its range is  $R(P_\nu) =T_\nu$, and we have $P_\nu^2=P_\nu=P_\nu^*$. We also denote by $I-P_\nu$ the orthogonal projection onto $\rm{span}\{\nu\}$. One has
$$
T_\bG(\nu)=T_\nu\times\R^k=\{(\hat z,\sigma)\in \bG\mid \hat z\in\R^{m},\ \sigma\in \R^k, \mbox{ such that } P_\nu \hat z=\hat z\}.
$$
From the expression in \eqref{uffa} we obtain for any $\hat{g}=(\hat z,\sigma)\in T_\bG(\nu)$
\begin{align}\label{uffa2}
& p(\hat{g},e,1) = 2^k (4\pi )^{-(\frac m2 + k)}  \int_{\R^k} e^{-i\langle \sigma,\la\rangle } \left(\det j(\sqrt{A(\la)})\right)^{1/2} 
\\
& \times \exp\bigg\{-\frac{1}{4t}\langle j(\sqrt{A(\la)}) \cosh \sqrt{A(\la)}\hat z,\hat z\rangle \bigg\}
d\la.\notag
\end{align}
Keeping in mind that for any $\la\in\R^k$ we have $j(\sqrt{A(\la)}) \cosh \sqrt{A(\la)}\in \operatorname{End}(\Rm)$, we introduce the notation 
\[
Q_\nu(\la) \overset{def}{=} P_\nu\  j(\sqrt{A(\la)}) \cosh \sqrt{A(\la)}\ P_\nu :T_\nu \rightarrow T_\nu,
\]
{\allowdisplaybreaks
and henceforth identify such map with the invertible and symmetric $(m-1)\times (m-1)$ matrix associated with it. 
Having fixed such notations, from \eqref{phi} and \eqref{uffa2} above we find
(after first making the change of variable $\sigma=2\pi \tau$, and then $\eta= \sqrt{Q_\nu(\la)}\ \hat z$, and subsequently using the well-known formula $\int_{\R^N} e^{-|\zeta|^2} d\zeta =\pi^{\frac{N}{2}}$)
\begin{align}\label{yuppi}
\phi_\bG(\nu) & =\frac{2^k}{\left(4\pi\right)^{\frac m2 +k}}\int_{\R^k}\int_{\R^k}\int_{T_\nu} e^{-i \langle\sigma,\la\rangle} \left(\det j(\sqrt{A(\la)})\right)^{1/2}
\\
& \times  \exp\bigg\{-\frac{1}{4}\langle j(\sqrt{A(\la)}) \cosh \sqrt{A(\la)} \hat z,\hat z \rangle \bigg\} d \hat z d\la d\sigma  \nonumber\\
&=\frac{1}{\left(4\pi\right)^{\frac m2}}\int_{\R^k}\int_{\R^k} e^{-2\pi i \langle\tau,\la\rangle} \left(\det j(\sqrt{A(\la)})\right)^{1/2} 
  \nonumber\\
& \times \int_{T_\nu} \exp\bigg\{-\frac{1}{4}\langle P_\nu j(\sqrt{A(\la)}) \cosh \sqrt{A(\la)} P_\nu \hat z,\hat z\rangle \bigg\} d \hat z d\la d\tau\nonumber
\\
&=\frac{1}{\left(4\pi\right)^{\frac m2}}\int_{\R^k}\int_{\R^k} e^{-2\pi i \langle\tau,\la\rangle} \frac{\left(\det j(\sqrt{A(\la)})\right)^{1/2}}{\left(\det Q_\nu(\la)\right)^{1/2}}\left(\int_{\R^{m-1}}  e^{-\frac{1}{4}|\eta|^2} d \eta\right) d\la d\tau\nonumber\\
&=\frac{1}{\sqrt{4\pi}}\int_{\R^k}\int_{\R^k} e^{-2\pi i \langle\tau,\la\rangle} \left(\frac{\det j(\sqrt{A(\la)})}{\det Q_\nu(\la)}\right)^{1/2} d\la d\tau\nonumber\\
&=\frac{1}{\sqrt{4\pi}}\int_{\R^k}  \hat f_\nu(\tau) d\tau\nonumber
\end{align}}
where we have let
\[
f_\nu(\la) = \left(\frac{\det j(\sqrt{A(\la)})}{\det Q_\nu(\la)}\right)^{1/2},
\]
and we are using the following definition of Fourier transform
\[
\hat f(\sigma) = \int_{\R^k} e^{-2\pi i\langle \sigma,\la\rangle} d\la.
\]
We now make the crucial observation that 
\[
A(0)=\mathbb{O}_m, \ \ \ \ j(\mathbb{O}_m)=\cosh(\mathbb{O}_m)=\mathbb{I}_m,
\]
and that $P_\nu$ is the identity on $T_\nu = R(P_\nu)$. These facts imply
\begin{equation}\label{heyhey}
f_\nu(0)=1.
\end{equation}
Moreover, keeping in mind that $\cosh x$ and $j(x) = \frac{x}{\sinh x}$ are even analytic functions on $\R$, and that $(\sqrt{A(\la)})^2=A(\la)\in C^{\infty}$, we have also $f_\nu\in C^\infty(\R^k)$. Furthermore, we notice that the injectivity of the Kaplan mapping $J$ ensures that $A(\la)$ is not the null endomorphism for every $\la\neq 0$ and, being $J(\la)$ skew-symmetric, the dimension of the range of $A(\la)$ has to be at least two. Hence, using the linearity of $J$ which allows us to write $\sqrt{A(\la)}=|\la|\sqrt{A(\la/|\la|)}$, one can deduce that there exists $k_0>0$ such that $\sqrt{A(\la)}$ has at least two eigenvalues growing bigger than $k_0|\la|$. This property accounts for the exponential decay of the functions $\la\mapsto j(\sqrt{A(\la)}), \cosh \sqrt{A(\la)}$, using which one can recognise that $f_\nu$ belongs in fact to the Schwartz class $\mathscr S(\R^k)$ (we mention that in Heisenberg type groups the function $f_\nu$ is independent of $\nu$ and it is given by $f_\nu(\la)=\left(\frac{1}{\cosh |\la|}\right)^n \left(\frac{|\la|}{\tanh |\la|}\right)^{1/2}$). We can then apply the inversion theorem for the Fourier transform and conclude from \eqref{heyhey} that
\[
1 = f_\nu(0) = \int_{\R^k} \hat f_\nu(\sigma) d\sigma,
\]
which completes, in view of \eqref{yuppi}, the proof of the desired claim in \eqref{wow}.



\bibliographystyle{amsplain}

\begin{thebibliography}{10}


\bibitem{Ad}
R. A. Adams, \emph{Sobolev Spaces}. Academic Press, New York, 1975.

\bibitem{ADM}
L. Ambrosio, G. De Philippis \& L. Martinazzi, \emph{Gamma-convergence of nonlocal perimeter functionals}. Manuscripta Math. \textbf{134}~(2011), no. 3-4, 377-403. 



\bibitem{BGG}
R. Beals, B. Gaveau \& P. Greiner, \emph{The Green function of model step two hypoelliptic operators and the analysis of certain tangential Cauchy Riemann complexes}. Adv. Math. \textbf{121}~(1996), no. 2, 288-345.


\bibitem{BBM1}
J. Bourgain, H. Brezis \& P. Mironescu, \emph{Another look at Sobolev spaces}. Optimal control and partial differential equations, 439-455, IOS, Amsterdam, 2001.

\bibitem{BBM2}
J. Bourgain, H. Brezis \& P. Mironescu,  \emph{Limiting embedding theorems for $W^{s,p}$ when $s\nearrow 1$ and applications}. Dedicated to the memory of Thomas H. Wolff. J. Anal. Math. \textbf{87}~(2002), 77-101. 

\bibitem{BMP}
M. Bramanti, M. Jr. Miranda \& D. Pallara, \emph{Two characterization of $\operatorname{BV}$ functions on Carnot groups via the heat semigroup}. Int. Math. Res. Not. IMRN \textbf{17}~(2012), 3845-3876. 

\bibitem{B}
H. Brezis, \emph{How to recognize constant functions. A connection with Sobolev spaces}. (Russian) Uspekhi Mat. Nauk 57 (2002), no. 4(346), 59--74; translation in Russian Math. Surveys 57 (2002), no. 4, 693-708.

\bibitem{BGT}
F. Buseghin, N. Garofalo \& G. Tralli, \emph{On the limiting behaviour of some nonlocal seminorms: a new phenomenon}. ArXiv: 2004.01303.

\bibitem{CRS}
L. Caffarelli, J.-M. Roquejoffre \& O. Savin, \emph{Nonlocal minimal surfaces}. Comm. Pure Appl. Math. \textbf{63}~(2010), no. 9, 1111-1144.

\bibitem{CV}
L. Caffarelli \& E. Valdinoci, \emph{Uniform estimates and limiting arguments for nonlocal minimal surfaces}. Calc. Var. Partial Differential Equations \textbf{41}~(2011), no. 1-2, 203-240. 


\bibitem{CDGcag}
L. Capogna, D. Danielli \& N. Garofalo, \emph{The geometric Sobolev embedding for vector fields and the isoperimetric inequality}. Comm. Anal. Geom. \textbf{2}~(1994), no. 2, 203-215.





\bibitem{CGr}
 L. Corwin \& F. P. Greenleaf, \emph{Representations of nilpotent Lie groups and their applications,
Part I: basic theory and examples}.
 Cambridge Studies in Advanced Mathematics 18,
Cambridge University Press, Cambridge
(1990).


\bibitem{CDKR}
M. Cowling, A. H. Dooley, A. Kor\'anyi \& F. Ricci, \emph{$H$-type groups and Iwasawa decompositions}. Adv. Math. \textbf{87}~(1991), no. 1, 1-41.

\bibitem{Cy}
J. Cygan, \emph{Heat kernels for class $2$ nilpotent groups}. Studia Math. \textbf{64}~(1979), no. 3, 227-238. 




\bibitem{Davila}
J. D\'avila, \emph{On an open question about functions of bounded variation}. Calc. Var. Partial Differential Equations \textbf{15}~(2002), no. 4, 519-527. 

\bibitem{DG54}
E. De Giorgi, \emph{Su una teoria generale della misura $(r-1)$-dimensionale in uno spazio ad $r$ dimensioni}. (Italian) Ann. Mat. Pura Appl. (4) \textbf{36}~(1954), 191-213. 


\bibitem{DT}
S. Dragomir \& G. Tomassini, 
\emph{Differential geometry and analysis on CR manifolds}. 
Progress in Mathematics, 246. Birkh\"auser Boston, Inc., Boston, MA, 2006. xvi+487 pp.

\bibitem{FFMMM}
A. Figalli, N. Fusco, F. Maggi, V. Millot \& M. Morini, \emph{Isoperimetry and stability properties of balls with respect to nonlocal energies}. Comm. Math. Phys. \textbf{336}~ (2015), no. 1, 441-507. 

\bibitem{Fofs}
G. B. Folland, \emph{A fundamental solution for a subelliptic operator}. Bull. Amer. Math. Soc. \textbf{79}~(1973), 373-376.


\bibitem{Fo}
G. B. Folland, \emph{Subelliptic estimates and function spaces on nilpotent Lie groups}. Ark. Mat. \textbf{13}~(1975), no. 2, 161-207. 


\bibitem{FShardy}
G. B. Folland \& E. M. Stein, \emph{Hardy spaces on homogeneous groups}. Mathematical Notes, \textbf{28}. Princeton University Press, Princeton, N.J.; University of Tokyo Press, Tokyo, 1982. xii+285 pp. 

\bibitem{FS}
R. L. Frank \& R. Seiringer, \emph{Non-linear ground state representations and sharp Hardy inequalities}. J. Funct. Anal. \textbf{255}~(2008), no. 12, 3407-3430. 


\bibitem{GParis}
N. Garofalo, \emph{Hypoelliptic operators and some aspects of analysis and geometry of sub-Riemannian spaces. Geometry, analysis and dynamics on sub-Riemannian manifolds}, Vol. 1, 123-257, EMS Ser. Lect. Math., Eur. Math. Soc., Z\"urich, 2016.


\bibitem{Glincei}
N. Garofalo, \emph{On the best constant in the nonlocal isoperimetric inequality of Almgren and Lieb}. Atti Accad. Naz. Lincei Rend. Lincei Mat. Appl. \textbf{31}~(2020), 465-470.

\bibitem{GNcpam}
N. Garofalo \& D. M. Nhieu, \emph{Isoperimetric and Sobolev inequalities for Carnot-Carath\'eodory spaces and the existence of minimal surfaces}. Comm. Pure Appl. Math. \textbf{49}~ (1996), no. 10, 1081-1144.

\bibitem{GTfi}
N. Garofalo \& G. Tralli, \emph{Functional inequalities for class of nonlocal hypoelliptic equations of H\"ormander type}. Nonlinear Anal. \textbf{193}~(2020), special issue `Nonlocal and Fractional Phenomena', 111567.

\bibitem{GTiso}
N. Garofalo \& G. Tralli, \emph{Nonlocal isoperimetric inequalities for Kolmogorov-Fokker-Planck operators}. J. Funct. Anal. \textbf{279}~(2020), 108591.

\bibitem{GTstep2}
N. Garofalo \& G. Tralli, \emph{Mehler met Ornstein and Uhlenbeck: the geometry of Carnot groups of step two and their heat kernels}.  Preprint 2020.

\bibitem{GTisoG}
N. Garofalo \& G. Tralli, \emph{Sharp nonlocal isoperimetric inequalities in Carnot groups}, preprint.

\bibitem{GTledG}
N. Garofalo \& G. Tralli, \emph{On the limit as $t\to 0^+$ of the sub-Riemannian Ledoux functional in the Heisenberg group}, preprint.

\bibitem{Gav}
B. Gaveau, \emph{Principe de moindre action, propagation de la chaleur et estim\'ees sous elliptiques sur certains groupes nilpotents}. Acta Math. \textbf{139}~(1977), no. 1-2, 95-153. 

\bibitem{JS}
D. S. Jerison \& A. S\'anchez-Calle, \emph{Estimates for the heat kernel for a sum of squares of vector fields}. Indiana Univ. Math. J. \textbf{35}~(1986), no. 4, 835-854. 

\bibitem{Ho}
L. H{\"o}rmander,
\textit{Hypoelliptic second order differential equations}. Acta Math. \textbf{119}~(1967), 147-171.

\bibitem{Hu}
A. Hulanicki, \emph{The distribution of energy in the Brownian motion in the Gaussian field and analytic-hypoellipticity of certain subelliptic operators on the Heisenberg group}. Studia Math. \textbf{56}~(1976), no. 2, 165-173. 

\bibitem{Ka}
A. Kaplan, \emph{Fundamental solutions for a class of hypoelliptic PDE generated by composition of quadratic forms}. Trans. Amer. Math. Soc. \textbf{258}~(1980), no. 1, 147-153.

\bibitem{KS}
S. Kusuoka \& D. Stroock, \emph{Long time estimates for the heat kernel associated with a uniformly subelliptic symmetric second order operator}. Ann. of Math. (2) \textbf{127}~(1988), no. 1, 165-189. 

\bibitem{Led}
M. Ledoux, \emph{Semigroup proofs of the isoperimetric inequality in Euclidean and Gauss space}. Bull. Sci. Math. \textbf{118}~(1994), no. 6, 485-510.


\bibitem{MP}
A. Maalaoui \& A. Pinamonti, \emph{Interpolations and fractional Sobolev spaces in Carnot groups}. Nonlinear Anal. \textbf{179}~(2019), 91-104.

\bibitem{Ma}
V. Maz'ya, \emph{Lectures on isoperimetric and isocapacitary inequalities in the theory of Sobolev spaces}. Heat kernels and analysis on manifolds, graphs, and metric spaces (Paris, 2002), 307-340, Contemp. Math., 338, Amer. Math. Soc., Providence, RI, 2003. 

\bibitem{MN}
V. Maz'ya \& J. Nagel, \emph{\"Uber \"aquivalente Normierung der anisotropen Funktionalr\"aume $H^\mu(\Rn)$}. (German) 
Beitr\"age Anal. No. 12 (1978), 7-17.

\bibitem{MS}
V. Maz'ya \& T. Shaposhnikova, \emph{On the Bourgain, Brezis, and Mironescu theorem concerning limiting embeddings of fractional Sobolev spaces}. J. Funct. Anal. 195~(2002), no. 2, 230-238.

\bibitem{MPPP}
M. Jr. Miranda, D. Pallara, F. Paronetto \& M. Preunkert, \emph{Short-time heat flow and functions of bounded variation in $\RN$}. Ann. Fac. Sci. Toulouse Math. (6) \textbf{16}~(2007), no. 1, 125-145.

\bibitem{NSW}
A. Nagel, E. M. Stein \& S. Wainger, \emph{Balls and metrics defined by vector fields. I. Basic properties}. Acta Math. \textbf{155}~(1985), no. 1-2, 103-147.

\bibitem{Ran}
J. Randall,  \emph{The heat kernel for generalized Heisenberg groups}. J. Geom. Anal. \textbf{6}~(1996), no. 2, 287-316.

\bibitem{Snice}
E. M. Stein, \emph{Some problems in harmonic analysis suggested by symmetric spaces and semi-simple groups}. Actes du Congr\`es International des Math\'ematiciens (Nice, 1970), Tome 1, pp. 173-189. Gauthier-Villars, Paris, 1971.



\bibitem{T1}
M.H. Taibleson, \emph{On the theory of Lipschitz spaces of distributions on Euclidean $n$-space. I. Principal properties}. J. Math. Mech. \textbf{13}~(1964), 407-479.

\bibitem{T2}
M.H. Taibleson, \emph{On the theory of Lipschitz spaces of distributions on Euclidean $n$-space. II. Translation invariant operators, duality, and interpolation}. J. Math. Mech. \textbf{14}~(1965), 821-839. 





\end{thebibliography}

\end{document}